\documentclass[11pt]{amsart}
\usepackage{geometry}                
\usepackage{amsmath}
\usepackage{graphicx}
\usepackage{amssymb}
\usepackage{epstopdf}
\usepackage[all,cmtip]{xy}
\usepackage{enumitem}
\usepackage{hyperref}

\newtheorem{teo}{Theorem}
\newtheorem{defi}[teo]{Definition}
\newtheorem{obs}[teo]{Remark}
\newtheorem{pro}[teo]{Proposition}
\newtheorem{cor}[teo]{Corollary}
\newtheorem{lem}[teo]{Lemma}

\newcommand{\spa}{\mathrm{span \,}}

\newcommand{\N}{\mathbb N}

\newcommand{\R}{\mathbb R}

\title{Non-ergodic Banach spaces are near Hilbert}
\author{W. Cuellar Carrera} \thanks{The author  was supported by FAPESP grant 2014/25900-7.}
\date{}
\address{Instituto de Matem\'atica e Estat\'istica, Universidade de S\~ao Paulo, R. do Mat\~ao 1010 SP-Brazil  \\ 
Email address:  \url{cuellar@ime.usp.br} }

\begin{document}
\maketitle
\begin{abstract} We prove that  a non  ergodic Banach space must be near Hilbert. In particular, $\ell_p$ ($2<p<\infty$) is ergodic.  This reinforces the  conjecture that  $\ell_2$ is the only non ergodic Banach space. As an application of our criterion for ergodicity, we prove that  there is no  separable Banach space which is complementably universal for the class of all subspaces of $\ell_p$, for $1\leq p <2$. This solves a question left open by  W. B. Johnson and  A. Szankowski in 1976.
\end{abstract}

\section{Introduction}
The solution of  Gowers \cite{Gd}  and  Komorowski--Tomczak-Jaegermann \cite{KTJ}  to the homogeneous Banach space problem,  provides that every Banach space  having  only one equivalence class for the relation of isomorphism between its infinite dimensional  subspaces must be isomorphic to $\ell_2$.  G. Godefroy  formulated the question  about  the number of non isomorphic subspaces of a Banach space  $X$ not isomorphic to $\ell_2$. 
This question was studied, in the context  of descriptive set theory, by
V. Ferenczi and C. Rosendal \cite{FR} who introduced  the notion of   {\it ergodic Banach space}   to study  the classification of the relative complexity  of the isomorphism relation between the subspaces of a separable  Banach space.

Our general reference for descriptive set theory will be the book  \cite{Ke}.
A {\it Polish space} is a separable topological space which admits a 
compatible complete metric. The  Borel sets of a Polish space is the $\sigma$-algebra generated by the  open sets. A set $X$ equipped with a $\sigma$-algebra  is called a {\it Borel standard space} if there exists a Polish topology  on $X$  for which  that $\sigma$-algebra  arises as  the collection of Borel subsets of $X$.   A function between two Borel standard spaces  $f: X \to Y$ is said to be  Borel, if $f^{-1}(B)$ is  Borel in $X$, for every  Borel subset $B\subseteq Y$.  

Given a  Polish space $X$,  let  $\mathcal F(X)$  be the collection of all closed subsets of $X$. The $\sigma$-algebra on $\mathcal F(X)$ generated by   
$$ A_{U}=\{ F\in  \mathcal F(X)\, : \, F\cap U\neq \emptyset \},
$$
where $U$ is an open subset of $X$, is called the  {\it Effros Borel structure} on the closed subsets of $X$. It is not hard to see that $\mathcal F(X)$ equipped  with this  Borel structure is a Borel standard space.  $\mathcal {SB}(X)$  denotes  the collection of  infinite dimensional linear subspaces $Y\in \mathcal F(X)$  equipped  with the relative Effros Borel structure.  This  framework allows us  to identify every class of subspaces of  a Banach space $X$  with a subset of $\mathcal{SB}(X)$ in which its complexity can be measured. Since $C(2^\N)$ is universal for all separable Banach spaces,  we have that properties of separable Banach spaces  becomes sets in $\mathcal{SB}(C(2^\N))$.   In \cite{B} it was  proved that the relation of isomorphism between separable Banach spaces is  an analytic and not borelian subset of $\mathcal{SB}(C(2^\N))^2$.
 
The central notion to study the complexity of analytic and Borel equivalence relations on Borel standard spaces is the concept of {\it Borel  reducibility}, which originated from the works  H. Friedman and L. Stanley \cite{FS} and independently from the works of L. A. Harrington, A. S. Kechris and A. Louveau  \cite{HKL}.
\begin{defi} Let  $R$ and $S$ be  two  Borel equivalence relations on Borel standard spaces $X$ and $Y$, respectively. One  says that $ R$  is  Borel reducible to $S$,   (denoted by $R\leq_{B} S$) if there exists a Borel function $\phi: X \to Y$ such that
$$ xRy \iff \phi(x)S\phi(y),
$$
for all $x,y\in X$.  The relation $R$ is  \emph{Borel bireducible} to $S$, (denoted by $R\sim_B S$), whenever both $R\leq_B S$ and $S\leq_B R$ hold. 
\end{defi}
This  can be interpreted as that the equivalence relation $R$ is classified by a Borel assignment  of invariants provided by  equivalence classes for $S$.
Observe that a Borel reduction induces an   embedding from the quotient space $X/R$ to $Y/S$, so $X/R$ has less than or equal   cardinality  that of $Y/S$. 

Ferenczi,  Louveau and Rosendal \cite{FLR}  proved that the relation of isomorphism between separable Banach spaces is a complete analytic equivalence relation, i.e., that any analytic equivalence relation Borel reduces to it. 

For $X$ a Polish space,   let id($X$) be the identity relation 
on the space $X$.  Since any two  standard Borel spaces with the same cardinality  are Borel isomorphic, it follows that for any uncountable $X$,
\[{\rm id}(X)\sim_B {\rm id}(\R).\] 
Among the uncountable  Borel equivalence relations,  the simplest  is ${\rm id}(\R)$. In fact, it was proved by  Silver \cite{Si}  that given a Borel equivalence relation $(X,R)$,  either it has countable many classes of equivalence,  or ${\rm id}(\R)$ is Borel reducible to  $(X, R)$. 
An equivalence relation admitting  the reals as a complete invariant is called  {\it smooth}, that is, when  it is reducible to ${\rm id}(\R)$.

The simplest example of a non-smooth equivalence relation is the relation of eventual agreement $E_0$ on $2^{\N}$, i.e.,  for $x,y\in 2^{\N}$, 
$$xE_0y \iff  (\exists N\in \N)  (x(n)=y(n), \,n\geq N).$$
Harrington, Kechris, and Louveau \cite{HKL} proved that $E_0$ is minimal  among non-smooth Borel equivalence relations with respect to  $\leq_B$.  

The following notion measures the complexity of the relation of isomorphism between subspaces of a separable Banach space and  was introduced by Ferenczi and Rosendal \cite{FR}.
\begin{defi} A separable Banach space $X$ is ergodic if  
$$(2^\N, E_0)\leq_B (\mathcal{SB}(X), \simeq).$$
\end{defi}
It follows that an ergodic Banach space  has at least $2^\N$ non-isomorphic subspaces and the equivalence relation of isomorphism between its subspaces is non-smooth. 

Rosendal  \cite{R} notices that  every hereditarily indecomposable (H.I) Banach space (i.e., a space in which no  closed infinite dimensional  subspace can be written as the direct sum of two closed infinite dimensional subspaces) is ergodic. By Gowers dichotomy \cite{Gd},  every Banach space  contains  an H.I subspace or an unconditional basic sequence. Since every Banach space containing an ergodic subspace must be ergodic,  one  can approach the study of ergodicity  by first restricting to spaces with unconditional basis.   

Ferenczi and Rosendal  \cite{FR} proved that a non ergodic Banach space  $X$ with unconditional basis  satisfies some regularity properties like being isomorphic to its square and to its hyperplanes, and more generally must be  isomorphic to $X\oplus Y$ for  any subspace  $Y$ of $X$ generated by a subsequence of the basis.  It was conjectured in \cite{FR} that every  separable Banach space not isomorphic to $\ell_2$ must be ergodic.

Dilworth, Ferenczi, Kutzarova and Odell  \cite{DFKO} proved  that every  Banach space $X$ with a strongly asymptotic $\ell_p$  basis  ($1\leq p \leq \infty$) not  equivalent to the unit vector basis of $\ell_p$ (or $c_0$ if $p=\infty$) is ergodic.
This result was generalized by R. Anisca \cite{A}, who  constructed  explicit  Borel reductions to prove that every separable asymptotically Hilbertian  space (and therefore every weak Hilbert space) not isomorphic to $\ell_2$ is ergodic.

Recall that a Banach space $X$ is called  \emph{(complementably) minimal} (notions due to Pe\l czy\'nski and Rosenthal, respectively) if every infinite-dimensional  closed subspace $Y$ of $X$ contains a  (complemented) subspace $Z$ isomorphic to $X$. Clearly, every (complemented) subspace of a (complementably) minimal space is also a (complementably) minimal space.  
Ferenczi \cite{F}   proved  that  a  separable Banach space without minimal subspaces must be ergodic. 
Hence, the conjecture in \cite{FR} is related to the following problem: {\emph Is every  minimal Banach space not isomorphic to $\ell_2$  ergodic?}.

It is well known that   $c_0$ and  $\ell_p$ ($1\leq p<\infty$) are complementably minimal spaces, while the dual of the Tsirelson space $T^*$  is an example of minimal but not complementably minimal space \cite{CJT}.   The first example of a complementably minimal space other than $c_0$ and  the $\ell_p$'s   is the Schlumprecht space and its dual \cite{Sc}. The list of minimal spaces known so far is completed with the family of Schlumprecht type spaces and their duals constructed  by  complex interpolation methods in \cite{CKKM}, and  every infinite dimensional closed subspace of each of the above.   For classical spaces,  it was proved in \cite{FG}  that  $c_0$ and $\ell_p$  for $1\leq p<2$ are ergodic. Rosendal  \cite{R2}   proved that the dual of the Tsirelson space is ergodic. 
In this work we  prove ergodicity for  a general family of Banach spaces including all the other minimal spaces not isomorphic to $\ell_2$  listed above.  More specifically,
given a  Banach space $X$, let 
\begin{align*}p(X)&= \sup\{p\, :\, \text{$X$ has type $p$} \}, \\
q(X)&= \inf\{q\, :\, \text{$X$ has cotype $q$}\}.
\end{align*}
Recall  that a  Banach space $X$ is said to be {\it near Hilbert}  when  $p(X)=q(X)=2$.  We give a criterion for ergodicity  which  together with the Johnson and Szankowski construction of subspaces without the approximation property allows  to prove that a non-ergodic Banach space must be near Hilbert.  In particular, we solve the question of \cite{FR} about the ergodicity of the $\ell_p$ spaces, for $p>2$. We also prove that the family of  Schlumprecht type spaces and its dual are not near Hilbert, and therefore they are ergodic spaces.

Finally, as an application  of the criterion for ergodicity, we prove that  for every non near Hilbert space $X$ there does not exist  separable Banach space which is complementably universal for the  class of all subspaces of $X$. 
In particular, this is true for $X=\ell_p$, $p \not=2$. This solves a problem left open by Johnson and Szankowki  in their 1976 paper \cite{JS} and mentioned again  in \cite{JLS}. (Johnson and Szankowki verified the case $2<p<\infty$ in \cite{JS})

\section{Criterion  for ergodicity}

A Banach space $X$ has the approximation property  (AP)  if the identity operator on $X$ can be approximated uniformly on compact subsets of $X$ by linear operators of finite rank.   The Banach space $X$ is said to have the bounded approximation property (BAP) if there exists $\lambda>0$ such that the finite rank operator $T$ in the definition of AP can be taken with norm $\|T\| \leq \lambda$. In 1973,  Enflo \cite{E} presented the first example  of Banach space without the AP and therefore without a  Schauder basis.   Enflo's construction was simplified by Davie \cite{D} who used probabilistic methods to construct such examples inside   $\ell_p$-spaces $(2<p\leq \infty)$.    Later in 1978, Szankowski \cite{S} proved that  the other range  of $\ell_p$-spaces ($1\leq p <2$)  also  have subspaces failing AP.  
 The criterion we  introduce  to  study  ergodicity in Banach spaces  is based on a criterion introduced by Enflo, and  used in the works of  Davie and Szankowski,  to prove that  a space  fails the AP.   

\

We first introduce some notation used throughout the paper. For every $n\in \N$,  denote by $I_n=\{ 2^n, 2^n+1, \ldots , 2^{n+1}-1\}$.  Given a Banach space $X$ and  sequences of vectors $(z_{n,\epsilon})_{n\in \N}$  in $X$,  $(z_{n,\epsilon}^*)_{n\in \N}$ in $X^*$,   $(\epsilon=0,1)$, we denote by  $Z=  \overline{span} \{ z_{j,  \epsilon}: \, j\in \N, \, \epsilon=0,1 \}$ and  we shall consider for every  $t\in 2^{\N}$  the closed subspace   $$X_t= \overline{\spa} \left \{ z_{j,t(n)} \, : \, j\in I_{n},  \, n=1,2,3, \ldots \right\}.$$   
If  $T: X_t \to Z$ is  a bounded and linear operator  we define the $n$-trace  of $T$ as
\[ \beta^n_{t}(T) = 2^{-n} \sum_{j\in I_n} z_{j,t(n)}^*T(z_{j,t(n)}). 
\]

\begin{defi} \label{cEnflo}A Banach space $X$ satisfies the  {\emph Cantorized-Enflo criterion}  if there exist  bounded  sequences of vectors $(z_{n,\epsilon})_{n\in \N}$  in $X$,  $(z_{n,\epsilon}^*)_{n\in \N}$ in $X^*$ $(\epsilon=0,1)$ and a sequence of   real scalars $(\alpha_n)_n$  such that
\begin{enumerate}
\item $z_{i,\epsilon}^*(z_{j,\tau})=\delta_{ij} \delta_{\epsilon \tau}$ for all $i,j\in \N$ and $\epsilon, \tau=0,1$.
\item \label{A}   For every $t,s\in 2^\N$ and every operator $T: X_t \to X_s$
$$\left | \beta^n_{t}(T) - \beta^{n-1}_{t}(T) \right | \leq \alpha_n\|T\|
$$
\item \label{B} $\sum_n  \alpha_n<\infty$.
\end{enumerate}
\end{defi}

Recall that a  subset of a topological space is said to be meagre if it is  the countable union of nowhere dense subsets (sets whose closure has empty interior).
An equivalence relation on a standard Borel space $X$ is said to be meagre if  it is a meagre subset of $X^2$.

Let $t\in 2^{\N}$ and $n\in \N$. We denote by $t/n=\{k\leq n: \, t(k)=1\}$.   The $E_0'$ equivalence relation on $2^\N$ is defined as 
\[ xE_0'y \, \iff \, \exists n (|t/n|=|s/n|) \wedge (t(k)=s(k), \, k\geq n).
\]
$E_0'$ is a refinement of $E_0$, that is $E_0'\subseteq E_0$. 
In connection with Borel reducibility ordering we shall use the following result from  Rosendal \cite{R}.

\begin{pro}\cite[Prop. 15]{R} \label{rosendal} Let $E$ be a meagre equivalence relation on $2^\N$ containing $E_0'$. Then $E_0\leq_B E$.
\end{pro}
Recall that a Banach space $X$ is said to be  \emph{ complementably universal} for a family $\mathcal A$ of Banach spaces if every space in $\mathcal A$ is isomorphic to a complemented subspace of $X$.   
In  \cite{JS}  Johnson and Szankowski  proved that there is no separable Banach space which is   complementably universal  for the class $\mathcal{A}_p$ of all subspaces of $\ell_p$, $2<p<\infty$. As observed by W. B. Johnson \cite{J},   it follows   that a complementably  universal  Banach space for the class $\mathcal{A}_p$ ($2<p<\infty$) must have density character  at least the continuum., where the density character of a topological space $X$ is the least cardinality of a dense subset of $X$.
In particular, this shows that  the family of non isomorphic subspaces of $\ell_p$, for $(2<p<\infty)$ has the cardinality of the continuum.  We use the ideas of the proof  in \cite{JS} (see also \cite{JS3}) to  establish   a criterion for ergodic Banach spaces.

\begin{lem} \label{uncoun} Let  $X$ be a Banach space satisfying the Cantorized-Enflo criterion,  and $\Gamma \subseteq 2^\N$.  Then  every Banach space which is  complementably universal for the family $\{X_t\}_{t\in\Gamma}$ has density character at least  cardinality of $\Gamma$.
\end{lem}

\begin{proof}
Let $X$ be  a Banach space satisfying the Cantorized-Enflo criterion and consider sequences $(z_{n,\epsilon})_{n\in \N}$  on $X$,  $(z_{n,\epsilon}^*)_{n\in \N}$ on $X^*$, ($\epsilon=0,1$),  and real scalars $(\alpha_n)_n$  as in  Definition \ref{cEnflo}. 
Suppose  that there is   $\Gamma \subseteq 2^{\N}$   uncountable  and a separable Banach space $W$ such that  for every  $t\in \Gamma$, $X_t$ is isomorphic to  a complemented  subspace of  $W$.  We may assume that for some $\lambda>0$  there exist  an embedding $T_t: X_t \to W$  and a projection onto $P_t:  X_t \to T_tX_t$, for every  $t\in \Gamma$,  such that $\|P_t\|\leq \lambda$ and 
$$ \|x\| \leq \|T_tx\| \leq \lambda \|x\|   \, \, \, \, \text{for every}  \, \, \ x\in X_t.
$$
Take $\delta>0$, it follows by conditions (\ref{A}) and  (\ref{B}) in Definition \ref{cEnflo}, that there exists $k=k(\delta)$ such that for every $m>k$
$$ \left |  \beta^m_t(T) -\beta^k_t(T) \right | \leq \delta \|T\|,
$$   
for every $t,s\in 2^\N$ and  any operator $T: X_t \to X_s$.
Since $W$ is separable and $\Gamma$ is uncountable,  there exist $t  \neq s \in \Gamma$ such that  $t(i)=s(i)$  ($i\leq k$) with 
$$ \|T_t (z_{j, t(k)}) - T_s(z_{j, s(k)})\| \leq 1/ \lambda 2^{k},   \, \, \, \, j\in I_{k}.
$$
Define now $T: X_t \to X_s$ by $T= T_s^{-1}P_sT_t$, where $T_s^{-1}: T_sX_s \to X_s$. We have  $T_s^{-1}P_s( T_t (z_{j, t(k)}) - T_s(z_{j, s(k)})) = T(z_{j, t(k)})- z_{j,s(k)}$,  and therefore  
$$ \sum_{j\in I_{k}} \| Tz_{j, t(k)}- z_{j, s(k)}\|  \leq \sum_{j\in I_{k}} \|T_s^{-1}P_s\| \| T_t (z_{j, t(k)}) - T_s(z_{j, s(k)})\| \leq 1.
$$
Hence we have, 
$$ \left |  \beta^k_t(T) \right | \geq 1- 2^{-k}   \sum_{j\in I_{k}}  \left \| z_{j,s(k)}^* \left (z_{j, s(k)}-Tz_{j,t(k)} \right)  \right \| \geq 1- 2^{-k}. 
$$   
Now since $t(m)\neq s(m)$  for some $m>k$ and $(z_{j,\epsilon}^*, z_{j,\epsilon})$ is a biorthogonal system,  we have 
$$ \beta^m_t(T)=0.
$$
Therefore, 
$$\|T\| \geq \delta^{-1}  \left |  \beta^m_t(T) -\beta^k_t(T) \right | \geq (1/2) \delta^{-1}.$$
On the other hand,
$$  \|T\|\leq  \|T_s^{-1}\| \|P_s\|\|T_t\| \leq \lambda^2.
$$
Since $\delta$ was arbitrary, we get a contradiction. 
\end{proof}

\begin{teo} \label{criterio} Every separable Banach space satisfying the Cantorized-Enflo criterion is ergodic.
\end{teo}
\begin{proof}

Let $X$ be a separable Banach space  satisfying the Cantorized-Enflo criterion. Define an equivalence relation $E$ on $2^{\N}$ by setting $sEt$ if and only if $X_s$ is isomorphic to $X_t$.  We observe that $E$ is $E_0'$-invariant. Indeed,  if  $tE_0' s$   then  $X_t$  and  $X_s$  are generated by the same sequence of vectors except for finite
sets of the same cardinality,
and therefore are isomorphic spaces. By  Lemma \ref{uncoun}  each equivalence class of $E$ is countable and then a meagre subset of $2^{\N}$. It is a general fact that an equivalence relation is meagre whenever each of its equivalence class is meagre \cite{Ke}. Hence $E$ is a meagre equivalence relation on $2^\N$, and  we have from  Proposition  \ref{rosendal} that $E_0 \leq_B E$. 
It is clear that  the  function $\phi: 2^{\N} \to \mathcal{SB}(X)$ given by $\phi(t)=X_t$ is Borel. In consequence, $X$ is ergodic.
\end{proof}

\begin{obs} A Banach space satisfying the Cantorized-Enflo criterion has a continuum of non isomorphic subspaces failing the bounded approximation property. 
\end{obs}

\begin{proof} We observe that the spaces $X_t$  used in the reduction fails the BAP for every $t\in 2^\N$. Given $\lambda>0$,  let  $n\in\N$  be such that $\lambda\sum_{k>n} \alpha_k\leq 1/2$. Let  $T: X_t \to X_t$ be an operator with  $\|T\|\leq \lambda$. Since  $|\beta^n_t(U)|\leq  \|U_{|Z_n}\|$ for every $U: X_t \to Z$,    where   $Z_n=\{z_{i,t(n)}, \, i\in I_n\}$ is a compact set, we have 
$$ \|(Id-T)_{|Z_n}\| \geq |\beta^n_t(Id-T)| \geq 1- |\beta^n_t(T)|\geq 1-\sum_{k >n} |\beta^k_t(T) - \beta^{k-1}_t(T)| \geq 1-\|T\|\sum_{k>n} \alpha_k >1/2. 
$$
\end{proof}
\begin{obs} \label{wAP}
Actually, if in Definition \ref{cEnflo}, we have for every $n\in \N$,  $\left | \beta^n_{t}(T) - \beta^{n-1}_{t}(T) \right | \leq \sup\{\|Tz\|, \, z\in F_n\}$ for a finite set  $F_n$ of vectors in $X$, such that    $\sum_n \sup\{\|z\|, \, z\in F_n\}<\infty$, then every $X_t$ fails the AP (\cite[Proposition 1]{S}).
\end{obs}

\begin{obs}
We also proved that $E_0$ is Borel reducible to the relation of  complemented biembeddability  between the subspaces of a  separable Banach space satisfying the Cantorized-Enflo criterion. 
\end{obs}

Let  $X$ and $Y$ be two Banach spaces and a constant  $K>0$. Recall that $X$ is said to be {\it $K$-crudely finitely representable} in $Y$ if  for every finite dimensional subspace $F$ of $X$  there exist a linear isomorphism  $T: F \to  T(F)\subseteq Y$  so that $\|T\|\|T^{-1}\|\leq K$.  $X$ is said to be {\it finitely representable} in $Y$ if $X$ is $(1+\epsilon)$-crudely finitely representable  in $Y$ for every $\epsilon>0$. A classical result of Maurey and Pisier  \cite{MP} states that $l_{p(X)}$ and $\ell_{q(X)}$ are finitely representable in $X$, for any Banach space $X$. The following Remark is stated in the classical book \cite{LT}.

\begin{obs} \label{finiterep}
It follows from the proof of  \cite[Theorem 1.a.5]{LT1} that if $\ell_p$ is $K$-crudely finitely representable in $Y$, for some   $1\leq p \leq \infty$, then Y has a subspace $X$  which has a Schauder decomposition into $\{X_n\}_{n=1}^{\infty}$ with $d(X_n, \ell_p^n)\leq K+1$ for every $n\in \N$.
\end{obs}

\begin{pro}\label{q>2} If $\ell_p$ is crudely  finitely representable in a Banach space $X$ for some $p>2$, then  $X$ satisfies the Cantorized-Enflo criterion.
\end{pro}

\begin{proof}  The proof  of  Johnson and Szankowski  \cite[Section IV]{JS}   that there does not exist separable Banach space which is complementably universal for the class of subspaces of $\ell_p$ ($2<p<\infty$) is  by modifying  Davie's construction of a subspace of $\ell_p$ ($2<p<\infty$) failing AP. 
We  observe that the Johnson and Szankowski construction 
yields that $\ell_p$ ($2<p<\infty$) satisfies the Cantorized-Enflo criterion.

Indeed, fix $p>2$. For every $n\in \N$,  we denote by $(f_j^n)_{j=1}^{3.2^n}$ the unit vector basis of $\ell_p^{3.2^n}$.  Using the notation of \cite{JS, D}, let  for $j\in I_n$ and $\epsilon=0,1$ 
$$z_{j,\epsilon}= e^{n+1}_{j + \epsilon 2^n},$$
$$z_{j,\epsilon}^*= \alpha^{n+1}_{j+ \epsilon 2^n},$$
where the vectors  $e^k_j$ defined in \cite{JS} have the form:
$$ e^k_j= \sum_{l=1}^{3.2^{k-1}} \lambda_j^k(l)f^k_l  + \sum_{l=1}^{3.2^k} \delta_j^k(l) f^k_l. 
$$
Also the functionals $\alpha^k_j$ are  linear combinations of the biorthogonal functionals $(f_l^{k^*})_{l=1}^{3.2^k}$ of $\ell_q^{3.2^k}$, satisfying $\alpha^k_l(e^i_j)=\delta_{ki}\delta_{lj}$. For any operator $T: X_t \to \ell_p$, 
$$ \left | \beta_t^n (T)- \beta_t^{n-1}(T) \right | \leq \sup \{ \|T\Phi^{k,t}_l\|, \, l\in F_n\},
$$
for some vectors $\Phi^{k,t}_l$ and a finite set $F_n$, where $\| \Phi^{k,t}_l\|\leq A (n+1)^{1/2}2^{-n(p-2)/2p}$, uniformly on $t$ and $l$.  Therefore $\ell_p$  satisfies the Cantorized-Enflo criterion.

Actually, we notice that  the previous construction only uses that $\ell_p$ has a natural Schauder decomposition into $\{\ell_p^{3.2^n}\}_{n=2}^{\infty}$. Therefore, if $\ell_p$ $(p>2)$ is crudely finitely representable in $X$,  then using  the  Remark \ref{finiterep},   there exist a constant $K>0$ and a subspace $Y$ of $X$ admitting a Schauder decomposition into $\{ X_n\}_{n=1}^{\infty }$, such that $d(X_n, \ell_p^{3.2^n})\leq K$. Hence, the analogous construction of vectors $e^k_j$ and $\alpha^k_j$ can be done as vectors supported in $X_{k-1}$ and $X_k$.
\end{proof}

\begin{cor} If  $\ell_p$ ($p>2$)  is crudely finitely representable in $X$, then $X$ is ergodic. 
\end{cor}
We observe that the construction of  Johnson and Szankowski \cite[Section IV]{JS} satisfies   the Cantorized-Enflo criterion in the form of  Remark  \ref{wAP},  so  each  of the $X_t$ constructed fails the AP.

\section{Case $p(X)<2$}

In this section we prove ergodicity for separable Banach spaces such that $p(X)<2$.  The particular case for the $\ell_p$'s spaces ($1\leq p <2$) was proved by Ferenczi and Galego \cite{FG},  where they actually reduce the relation $E_{K_\sigma}$ and use only subspaces with unconditional bases. Their approach relies on certain lower estimates on successive
vectors which have no reason to hold in the case when $\ell_p$ is only crudely finitely representable on $X$.

Our approach is to obtain the `Cantorized version'  of the subspaces of $\ell_p$ ($1\leq p<2$)  without AP constructed by Szankowski \cite{S}. The advantage of this method is that  the nature of that construction  allows to pass the Cantorized-Enflo criterion from $\ell_p$ to  a Banach space  $X$ for which $\ell_p$ is crudely finitely representable in $X$. 

Before the proof, we need  to define the following functions $f_k:\N \to \N$, $k\leq 8$, $g_k:\N \to \N$, $k\leq 15$, $h_k:\N \to \N$, $k\leq 32$  to encode the support of some vectors used in that construction. 
The main difference with \cite{S} is that our construction uses  vectors with support  of length  twelve  instead of six of the original one.  

\

$f_k(16i+l)= 8i+k-1$, $i= 2,3,4, \ldots$  $\; \; \; 0\leq l \leq 15$, $\; \; \;  1\leq k\leq 8$

$g_k(16i+l)= 16i+ (l+k) \, mod \, 16$, \, \, \, \,  $i= 2,3,4, \ldots$  $\; \; 0\leq l \leq 15$, $\; \; 1\leq k \leq 15$

$h_k(16i+l)= 32i+k-1$, $i= 2,3,4, \ldots$  $\; \; \; 0\leq l \leq 3$, $\; \; \; 1\leq k\leq 32$

\

We denote  by $I_n^j=\{k\in I_n \, : \, k\cong j \, (mod \, 16)\}$, $j=0,1,2\ldots, 15$. The following is a modified version  of  the key  Szankowski  combinatorial argument \cite{S} (See also  \cite[Prop 1.g.5 ]{LT}) adapted to our set of functions $\{f_k,\,  g_k, \, h_k\}$.
\begin{lem}\label{partition} There exist partitions $\Delta_n$ and $\nabla_n$ of $I_n$  into disjoint sets and a sequence of integers $(m_n)_n$ with $m_n\geq 2^{n/32-1}$, $n=2,3,\ldots$ such that

\begin{enumerate}
\item For every $A\in \nabla_n$, $m_n\leq |A| \leq 2m_n$ and it is contained in some $I_n^j$.
\item For every $A \in \nabla_n$ and every $B\in \Delta_n$, $|A\cap B| \leq 1$.
\item For every $A\in \nabla_n$ and every function $\xi$  in  $\{f_k,\,  g_k, \, h_k\}$,  the set $\xi(A)$  is contained entirely in an element of $\Delta_{n-1}, \Delta_n$ or $\Delta_{n+1}$.
\end{enumerate}

\end{lem}

\begin{proof} 
Consider the functions $\varphi^j_n: I_n^0 \to I_n^j$  given by $\varphi_n^j(k)=k+j$. For $n\geq 4$ and $r=0,1$ we let $\psi_n^r: I_n^0 \to I_{n+1}^0$ the map defined by $\psi_n^r(k)=2k+16r$. The above functions are 1-1 and have disjoint ranks with $I_{n+1}^0= \psi_n^0(I_n^0)\cup \psi_n^1(I_n^0)$.

Inductively, for $n\geq 4$ we can represent $I_n^0$ as the cartesian product  $C_n\times D_n$, where $|D_{n+1}|=|C_n|$, $|C_{n+1}|=2|D_n|$ and such that

\begin{enumerate}
\item For every $c\in  C_{n+1}$ there exist $d\in D_n$ and $r=0,1$ such that $\psi_n^r( C_n\times \{d\}) = \{c\} \times D_{n+1}$. 

\item For every $d\in D_{n+1}$ there exists $c\in C_n$ such that $\psi_n^0\cup \psi_n^1 (\{c\}\times D_n)= C_{n+1}\times \{d\}$. 
\end{enumerate}
This means that the functions $\psi_n^r$ send columns of $C_n\times D_n$ onto rows of $C_{n+1}\times D_{n+1}$ in a way that every column of $C_{n+1}\times D_{n+1}$ is the image of a row of $C_n\times D_n$ by  $\psi_n^0\cup \psi_n^1$.  Notice that $|C_n|, |D_n| \geq 2^{n/2-2}$. 

Now  we split each $D_n$ as a cartesian product of sixteen  factors $D_n = \prod_{l=0}^{15} D_n^l$ such that $$|D_n^0| \leq |D_n^1|\leq \ldots \leq |D_n^{15}|\leq 2|D_n^0|$$

The partitions are then defined as 

$$ \nabla_n = \left \{  \varphi_n^l  \left (\{c\} \times D_n^l \right) \, : \, c\in C_n\times \prod_{i\neq l} D_n^i, \, 0\leq l \leq 15\right \},
$$

$$ \Delta_n= \left \{ \varphi_n^l \left(C_n \times \prod_{i\neq l} D_n^i \times \{d\}\right) \, : \, d\in D_n^l, \, 0\leq l \leq 15 \right \}.
$$
The conditions 1), 2) and 3) are satisfied in the same way as \cite{S}
\end{proof}

\

\begin{teo} \label{p<2} If $\ell_p$  is  crudely finitely  representable in a Banach space $X$, for some $1\leq p<2$, then $X$  satisfies the Cantorized-Enflo criterion. 
\end{teo}

\begin{proof}
Let $X$ be a Banach space such that $\ell_p$ is crudely finitely representable, for some $1\leq p<2$. For every $n\in \N$, we fix $\Delta_n$ and $\nabla_n$  partitions of $I_n$  obtained by Lemma \ref{partition}.   It follows by Remark \ref{finiterep} that there exist a constant $K>0$ and a subspace $Y$ of $X$ admitting a Schauder decomposition into $\{X_n\}_{n=1}^{\infty}$ such that $d(X_n, \ell_p^{2^n})\leq K$, for every $n\in \N$.
Let $(x_j)_{j=1}^{\infty}$ be a bounded sequence of vectors in $Y$ with  $x_j \in X_n$ when $j\in I_n$  such that  for  every $n$
\begin{equation} \label{normv} K^{-1} \left(  \sum_{B\in \Delta_n} \left ( \sum_{j\in B} |a_j|^2 \right )^{p/2}\right )^{1/p} \leq \left \| \sum_{j\in I_n} a_jx_j \right \| \leq K \left(  \sum_{B\in \Delta_n} \left ( \sum_{j\in B} |a_j|^2 \right )^{p/2}\right )^{1/p},
\end{equation}
for any sequence of scalars $(a_j)_{j=1}^{\infty}$.  Let $(x_j^*)_{j=1}^{\infty}$ be a sequence of functionals in $Y^*$, such that $x_j^*(x_i)=\delta_{ij}$  for all $i,j\in \N$ and 
\begin{equation}  \label{normf} K^{-1} \left(  \sum_{B\in \Delta_n} \left ( \sum_{j\in B} |b_j|^2 \right )^{q/2}\right )^{1/q} \leq \left \| \sum_{j\in I_n} b_jx_j^* \right \|_{Y^*} \leq K \left(  \sum_{B\in \Delta_n} \left ( \sum_{j\in B} |b_j|^2 \right )^{q/2}\right )^{1/q}
\end{equation}
for any sequence of scalars $(b_j)_{j=1}^{\infty}$ and every $n$, where $1/p+1/q=1$. \\
We now define  the  sequence of vectors  $(z_{i, \epsilon})_i$, $\epsilon=0,1$  in $Y$ by setting:
$$
z_{i,0}  =  (x_{8i}- x_{8i+1}) +  (x_{8i+2}- x_{8i+3})+ x_{16i}+ x_{16i+1}+ x_{16i+4}+x_{16i+5}+x_{16i+8}+x_{16i+9}+x_{16i+12}+x_{16i+13}
$$
$$
z_{i,1}  =  (x_{8i+4}- x_{8i+5}) +  (x_{8i+6}- x_{8i+7})+ x_{16i+2}+ x_{16i+3}+ x_{16i+6}+x_{16i+7}+x_{16i+10}+x_{16i+11}+x_{16i+14}+x_{16i+15}
$$

Recall that $Z=  \overline{span} \{ z_{j,  \epsilon}: \, j\in \N, \, \epsilon=0,1 \}$.
Notice that  for every $i\in \N$,
\begin{align*}
(x_{8i}^*- x_{8i+1}^*)_{|Z}=  (x_{8i+2}^*- x_{8i+3}^*)_{|Z} &= 1/2(x_{16i}^*+x_{16i+1}^*+ x_{16i+8}^*+x_{16i+9}^*)_{|Z} \\
&=1/2(x_{16i+4}^*+x_{16i+5}^*+ x_{16i+12}^*+ x_{16i+13}^*)_{|Z}.
\end{align*}
Indeed,  all  four formulas give 2 when evaluated on $z_{i,0}$ and give 0 when evaluated on $z_{j,\epsilon}\neq z_{i,0}$. Analogously, for every $i\in \N$,
\begin{align*}
(x_{8i+4}^*- x_{8i+5}^*)_{|Z}=  (x_{8i+6}^*- x_{8i+7}^*)_{|Z} &= 1/2(x_{16i+2}^*+x_{16i+3}^*+ x_{16i+10}^*+x_{16i+11}^*)_{|Z} \\
&=1/2(x_{16i+6}^*+x_{16i+7}^*+ x_{16i+14}^*+ x_{16i+15}^*)_{|Z}.
\end{align*}
All  four formulas above give 2 when evaluated on $z_{i,1}$ and 0 when evaluated on $z_{j,\epsilon}\neq z_{i,1}$.
We define the sequence of functionals $(z_{n, \epsilon}^*)_{n\in \N}$, $\epsilon=0,1$ on $Z^*$ by setting
$$ z_{i,\epsilon}^*= 1/2 (x_{8i+4\epsilon}^*- x_{8i+4\epsilon+1}^*)_{|Z}. 
$$
Hence, 
\begin{align*}z_{i,0}^*=1/2(x_{8i+2}^*- x_{8i+3}^*)_{|Z}&= 1/4(x_{16i}^*+x_{16i+1}^*+ x_{16i+8}^*+x_{16i+9}^*)_{|Z}\\
&= 1/4(x_{16i+4}^*+x_{16i+5}^*+ x_{16i+12}^*+ x_{16i+13}^*)_{|Z}, 
\end{align*}
\begin{align*}z_{i,1}^*=1/2 (x_{8i+6}^*- x_{8i+7}^*)_{|Z} &= 1/4(x_{16i+2}^*+x_{16i+3}^*+ x_{16i+10}^*+x_{16i+11}^*)_{|Z} \\
&=1/4(x_{16i+6}^*+x_{16i+7}^*+ x_{16i+14}^*+ x_{16i+15}^*)_{|Z}.
\end{align*}
For $t\in 2^\N$,  recall that  $X_t= \overline{span} \{ z_{j,  t(n)} \, : \, j\in I_n, \, n\in \N \} $.
If $T:X_t \to Z$ is a linear and bounded operator, the $n$-trace of $T$ has been defined as
\[ \beta^n_{t}(T) = 2^{-n} \sum_{j\in I_n} z_{j,t(n)}^*T(z_{j,t(n)}). 
\]

We need to verify that the $\beta_n's$ satisfy the conditions of  the Cantorized-Enflo criterion (Definition \ref{cEnflo}).

\

CASE 1: $t(n)=t(n-1)=0$.
\begin{align*}   &\beta^n_t(T)- \beta^{n-1}_t(T) =2^{-n} \sum_{i\in I_n} z_{i,0}^* T(z_{i,0})- 2^{-n+1} \sum_{i\in I_{n-1}} z_{i,0}^* T(z_{i,0}) \\
			&=2^{-n} \sum_{i\in I_n} 2^{-1}  \left (  x_{8i}^*-x_{8i+1}^*\right ) T \left ( z_{i,0} \right)  - 2^{-n+1}\sum_{i\in I_{n-1}} 2^{-2} (x_{16i}^*+ x_{16i+1}^* + x_{16i+8}^*+ x_{16i+9}^*) T \left (z_{i,0} \right )
\end{align*}
		  $$   = 2^{-n-1}  \sum_{i\in I_{n-1}}  \{  x_{16i}^* T (z_{2i,0}-z_{i,0}) 
		   + x_{16i+1}^*T \left ( -z_{2i,0}-z_{i,0} \right )$$$$
		   + x_{16i+8}^* T(  z_{2i+1,0}- z_{i,0})
	         + x_{16i+9}^* T(-z_{2i+1,0} -z_{i,0}) \}$$
		     
\

The elements in parentheses above will be called $y_{16i}, y_{16i+1}, y_{16i+8}, y_{16i+9}$ respectively, thus 
$$ \beta^n_t(T)- \beta^{n-1}_t(T) = 2^{-n-1} \sum_{j\in I_{n+3}(0,0)} x_j^*T(y_j),
$$ 
where $I_{n}(0,0)=I_n^0\cup I_n^1\cup I_n^{8} \cup I_n^9$.

\

CASE 2:  $t(n)=0$, $t(n-1)=1$

\begin{align*}   &\beta^n_t(T)- \beta^{n-1}_t(T) =2^{-n} \sum_{i\in I_n} z_{i,0}^* T(z_{i,0})- 2^{-n+1} \sum_{i\in I_{n-1}} z_{i,1}^* T(z_{i,1}) \\
			&=2^{-n} \sum_{i\in I_n} 2^{-1}  \left (  x_{8i+2}^*-x_{8i+3}^*\right ) T \left ( z_{i,0} \right)  - 2^{-n+1}\sum_{i\in I_{n-1}} 2^{-2} (x_{16i+2}^*+ x_{16i+3}^* + x_{16i+10}^*+ x_{16i+11}^*) T \left (z_{i,1} \right )
\end{align*}
$$   = 2^{-n-1}  \sum_{i\in I_{n-1}}  \{  x_{16i+2}^* T (z_{2i,0}-z_{i,1}) 
		   + x_{16i+3}^*T \left ( -z_{2i,0}-z_{i,1} \right )$$$$
		   + x_{16i+10}^* T(  z_{2i+1,0}- z_{i,1})
	         + x_{16i+11}^* T(-z_{2i+1,0} -z_{i,1}) \}$$
		     
\

The elements in parentheses above will be called $y_{16i+2}, y_{16i+3}, y_{16i+10}, y_{16i+11}$ respectively, thus 
$$ \beta^n_t(T)- \beta^{n-1}_t(T) = 2^{-n-1} \sum_{j\in I_{n+3}(0,1)} x_j^*T(y_j),
$$ 
where $I_{n}(0,1)=I_n^2\cup I_n^3\cup I_n^{10} \cup I_n^{11}$.

\

CASE 3:  $t(n)=1$, $t(n)=0$
\begin{align*}   &\beta^n_t(T)- \beta^{n-1}_t(T) =2^{-n} \sum_{i\in I_n} z_{i,1}^* T(z_{i,1})- 2^{-n+1} \sum_{i\in I_{n-1}} z_{i,0}^* T(z_{i,0}) \\
			&=2^{-n} \sum_{i\in I_n} 2^{-1}  \left (  x_{8i+4}^*-x_{8i+5}^*\right ) T \left ( z_{i,1} \right)  - 2^{-n+1}\sum_{i\in I_{n-1}} 2^{-2} (x_{16i+4}^*+ x_{16i+5}^* + x_{16i+12}^*+ x_{16i+13}^*) T \left (z_{i,0} \right )
\end{align*}
		  $$   = 2^{-n-1}  \sum_{i\in I_{n-1}}  \{  x_{16i+4}^* T (z_{2i,1}-z_{i,0}) 
		   + x_{16i+5}^*T \left ( -z_{2i,1}-z_{i,0} \right )$$$$
		   + x_{16i+12}^* T(  z_{2i+1,1}- z_{i,0})
	         + x_{16i+13}^* T(-z_{2i+1,1} -z_{i,0}) \}$$
		     
\

The elements in parentheses above will be called $y_{16i+4}, y_{16i+5}, y_{16i+12}, y_{16i+13}$ respectively, thus 
$$ \beta^n_t(T)- \beta^{n-1}_t(T) = 2^{-n-1} \sum_{j\in I_{n+3}(1,0)} x_j^*T(y_j),
$$ 
where $I_{n}(1,0)=I_n^4\cup I_n^5\cup I_n^{12} \cup I_n^{13}$.

\

CASE 4: $t(n)=1$, $t(n-1)=1$
\begin{align*}   &\beta^n_t(T)- \beta^{n-1}_t(T) =2^{-n} \sum_{i\in I_n} z_{i,1}^* T(z_{i,1})- 2^{-n+1} \sum_{i\in I_{n-1}} z_{i,1}^* T(z_{i,1}) \\
			&=2^{-n} \sum_{i\in I_n} 2^{-1}  \left (  x_{8i+6}^*-x_{8i+7}^*\right ) T \left ( z_{i,1} \right)  - 2^{-n+1}\sum_{i\in I_{n-1}} 2^{-2} (x_{16i+6}^*+ x_{16i+7}^* + x_{16i+14}^*+ x_{16i+15}^*) T \left (z_{i,1} \right )
\end{align*}
		  $$   = 2^{-n-1}  \sum_{i\in I_{n-1}}  \{  x_{16i+6}^* T (z_{2i,1}-z_{i,1}) 
		   + x_{16i+7}^*T \left ( -z_{2i,1}-z_{i,1} \right )$$$$
		   + x_{16i+14}^* T(  z_{2i+1,1}- z_{i,1})
	         + x_{16i+15}^* T(-z_{2i+1,1} -z_{i,1}) \}$$
		     
\

The elements in parentheses above will be called $y_{16i+6}, y_{16i+7}, y_{16i+14}, y_{16i+15}$ respectively, thus 
$$ \beta^n_t(T)- \beta^{n-1}_t(T) = 2^{-n-1} \sum_{j\in I_{n+3}(1,1)} x_j^*T(y_j),
$$ 
where $I_{n}(1,1)=I_n^6\cup I_n^7\cup I_n^{14} \cup I_n^{15}$.

\

Hence, 

$$ \beta^n_t(T)- \beta^{n-1}_t(T)=  2^{-n-1} \sum_{j\in I_{n+3}(t(n), t(n-1)) } x_j^*T(y_j).
$$

We use the functions $\{f_k, \, g_k, \, h_k\}$ to describe the support of the vectors $y_j$. For each $j$, we shall need four functions of the $\{f_k, \, k\leq 8\}$, nine functions of the $\{g_k, \, k\leq 15\}$ and eight functions of the $\{h_k, \, k\leq 32\}$.  In fact, notice that 
$$y_j=  \sum_{k=1}^4 \alpha_{j_k}x_{f_{j_k}(j)}+  \sum_{t=1}^9 \beta_{j_t}x_{g_{j_t}(j)} + \sum_{s=1}^{8}  \gamma_{j_s}x_{ h_{j_s}(j)},$$
where  $|\alpha_{j,k}| =|\gamma_{j_s}|=|\beta_{j_t}|=1$ for all the indexes in the formula above, except for one $j_{t_0}$ which satisfies $|\beta_{j_{t_0}}|=2$.

Given $\epsilon, \delta=0,1$, we write $\nabla_n(\epsilon, \delta)= \{A \in \nabla_n \, : \, A\subseteq I_n(\epsilon,\delta)\}$. Notice that 

\begin{align*}2^{-n-1}\sum_{j\in I_{n+3}(\epsilon,\delta) } x_j^*T(y_j) &= 2^{-n-1} \sum_{A\in \nabla_{n+3}(\epsilon, \delta)} \, \,  \sum_{j\in A} x_j^*T(y_j)\\
&= 2^{-n-1} \sum_{A\in \nabla_{n+3}(\epsilon, \delta)}  2^{-|A|} \sum_{\theta}  \left ( \sum_{j\in A} \theta_j x_j^* \right ) \left (  \sum_{j\in A} \theta_j Ty_j \right ),
\end{align*}
where the sum is taken over all the choices of signs $\{\theta_j\}_{j\in A}$.  Observe that by Lemma \ref{partition}-(2) and the  equation (\ref{normf})  above  (about norm of the functionals $x_j^*$'s) we have, for every $A\in \nabla_{n+3}(\epsilon, \delta)$ and $\{\theta_j\}_{j\in A}$,

\begin{eqnarray*} \left \| \sum_{j\in A} \theta_j x_j^* \right \|_{Y^*}  &\leq& K \left(  \sum_{B\in \Delta_n} \left ( \sum_{j\in B\cap A} |\theta_j|^2 \right )^{q/2}\right )^{1/q}\\
&=& K|A|^{1/q}   \leq K(2m_{n+3})^{1/q}, 
\end{eqnarray*}
where $1/p + 1/q =1$.  By Lemma \ref{partition}-(3) we have for every $A\in \nabla_{n+3}(\epsilon, \delta)$,  $\{\theta_j\}_{j\in A}$ and  any function $\xi$ in $\{f_k, \, g_k, \, h_k\}$

$$  \left \| \sum_{j\in A} \theta_j x_{\xi(j)} \right \|\leq K |A|^{1/2}\leq K(2m_{n+3})^{1/2}.$$
It follows that

$$ \left \| \sum_{j\in A} y_j \right \| = \left \|  \sum_{j\in A} \sum_{k=1}^{21} \lambda_{j_k} x_{\xi_{l_k}(j)}  \right \| \leq 42K|A|^{1/2}\leq 42K(2m_{n+3})^{1/2}.
$$
Notice that by construction $|\nabla_n(\epsilon, \delta)|= 2^{-2}|\nabla_n|\leq 2^{n-2}m_n^{-1}$. Hence, 

\begin{align*} \left |  \beta^n_t(T)- \beta^{n-1}_t(T) \right |&=
 \left | 2^{-n-1}\sum_{j\in I_{n+3}(\epsilon, \delta) } x_j^*T(y)  \right | \\
 &\leq  2 ^{-n-1} (2^{n+1}m_{n+3}^{-1})K(2m_{n+3})^{1/q}42K\sqrt 2m_{n+3}^{1/2} \|T\| \\
&\leq  84K^2(m_{n+3})^{1/q+1/2-1} \|T\|
\end{align*}
Since $\alpha=1/2+1/q-1=1/2-1/p<0$, the series $\sum_n m_n^{\alpha}\leq \sum_n 2^{\alpha(n/32-1)}<\infty$. Therefore,  $X$  satisfies the Cantorized-Enflo criterion. 
\end{proof}

\begin{obs} The proof of the Theorem \ref{p<2} is based on the idea from \cite{S} where subspaces of $\ell_p$ ($1\leq p <2$) without AP were constructed.   It was pointed out by Szankowsi \cite{S} (see also \cite[Remark 2. pg 111]{LT})  that the mentioned idea can be easily adapted to obtain subspaces of $\ell_p$  ($2<p<\infty$)  without AP.  This implies that also  the method  used in the proof of Theorem \ref{p<2} is  valid  for Banach spaces  $X$ in which  $\ell_p$  ($2<p<\infty$) is crudely finitely representable. Indeed, the same definition of vectors $z_{i,\epsilon}$ and functionals $z_{i,\epsilon}^*$ works, it is only necessary to  modify the construction of the partitions $\Delta_n$ and $\nabla_n$ in Lemma \ref{partition}. This gives us an independent proof of Proposition \ref{q>2}.
\end{obs}

We observe that the construction of Theorem \ref{p<2} satisfies the Cantorized-Enflo criterion in the form of Remark \ref{wAP}, therefore  every $X_t$  constructed above fails the AP. We can conclude that: 
\begin{teo} \label{main}Every separable Banach space non near Hilbert satisfies the Cantorized-Enflo criterion and therefore is ergodic. Furthermore,  the reduction uses subspaces without  AP.
\end{teo}

The following remark is due to R. Anisca. 
\begin{obs}  There do  exist near-Hilbert spaces satisfying the Cantorized-Enflo criterion.  Indeed, Casazza, Garc\'ia and Johnson  \cite{CGJ} constructed an asymptotically Hilbertian space which fails the AP. Their approach follows closely the Davie construction \cite{D} of a subspace of $\ell_p=(\sum_n \ell_p^{3.2^n})_p$ ($2<p<\infty$) failling AP.   The space in \cite{CGJ} is instead  a subspace of $Z=(\sum_n \ell_{p_n}^{3.2^n})_2$  where $p_n \downarrow 2$ appropriately.  One can  combine  the arguments of   Proposition \ref{q>2} and  those in  \cite[Section IV]{JS}  to  construct a version  of the Casazza, Garc\'ia and Johnson space satisfying the Cantorized-Enflo criterion.
Also, the arguments from Theorem \ref{main} can be used in the context of construction by Anisca and Chlebovec \cite{AC} and obtain that spaces of the form $\ell_2(X)$, with $X$ of cotype 2 and having the sequence of Euclidean distances of order at least $(\log n)^{\beta}$ ($\beta>1$), satisfy the Cantorized-Enflo criterion. 

\end{obs} 

As a  direct consequence of  Lemma \ref{uncoun} and Theorem \ref{p<2}, 
we can now extend  the result of Johnson and Szankowski \cite{JS} about complementably universal spaces for the family of subspaces of $\ell_p$ ($2<p<\infty$).
\begin{teo}  There is no   separable Banach space which is complementably universal  for the class of all subspaces of $X$ when $X$ is non near Hilbert. 
\end{teo}

\begin{cor} There is no  separable Banach space which is complementably universal for the family of subspaces of $\ell_p$ ($1\leq p <2$).
\end{cor}

In the limit case, Johnson and Szankowski \cite{JS2} constructed a separable space, non isomorphic to the Hilbert, such that all subspaces have the BAP, and which is
complementably  universal for all its subspaces. Also, if  every subspace of $X$ has BAP (for example if $X$ is weak Hilbert), then the  Pe\l czy\'nski's universal  space (see \cite[Theorem 2.d.10(a)]{LT1})   is complementably universal for the family of  all subspaces of $X$. 

\section{The Schlumprecht type space $S_{p,r}$ is  not near Hilbert}

Schlumprecht  \cite{AS,Sc}  constructed the  first example  of a complementably minimal  Banach space $S$ different from the classical spaces $c_0$ and $\ell_p$ ($1\leq p < \infty$).  In \cite{CKKM},   the Schlumprecht construction was extended to uniformly convex
examples using interpolation techniques. In fact, they constructed a
family of uniformly convex complementably minimal spaces by   interpolating $S$ and $\ell_q$.

The approach in \cite{CKKM} deals with Banach spaces $X$ defined  by lattice norms $\|.\|_X$ on $c_{00}$.  In this context, if $X$ and $Y$ are two such spaces and $0<\theta<1$,  then $X^{1-\theta}Y^{\theta}$ is defined as  the space $Z$ with the norm $\|z\|_Z= \inf\{ \| x\|_X^{1-\theta} \| y\|_Y^{\theta}, \, z=|x|^{1-\theta}|y|^{\theta}\}$.   When we consider the complex scalars and either $X$ or $Y$ is separable, then $Z$ coincides with the usual complex interpolation  space $[X,Y]_{\theta}$ (see \cite{C}).
\begin{defi} For every $1\leq p < r \leq \infty$, the  Schlumprecht type space $S_{p,r}$ is defined as the interpolated space $\ell_t^{1-\theta} S^{\theta}$, where $\theta= \frac{1}{p}-\frac{1}{r}$ and $t=(1-\theta)r$.
\end{defi}

\begin{pro}{\cite[Proposition 3 and Theorem 8]{CKKM}} \label{ScFamily}  For any $1\leq p < r \leq \infty$, the space $S_{p,r}$ and its dual  are complementably minimal.  Furthermore, $S_{p,r}$  has an 1-unconditional normalized basis $(e_n)_{n=1}^{\infty}$ such that $\| \sum_{i=1}^{n} e_i \|_{S_{p,r}} =n^{1/p}\log_2(n+1)^{1/r-1/p}$, for every $n\in \N$.
\end{pro}
Notice that $S$ is simply $S_{1,\infty}$.    We use the ideas from \cite{CFM} and estimates of the norm of  some combination of the vector basis to compute $p(S_{p,r})$ and $q(S_{p,r})$.  

\begin{pro} \label{ScCotype} Let $1\leq p < r \leq \infty$. For every $n\in \N$ and $\epsilon>0$, there exists a sequence of vectors $v_1, \ldots, v_n$ in $c_{00}$ such that 
\begin{enumerate}
\item The set of vectors $\{v_1, \ldots, v_n\}$ are  disjointly supported.
\item  $\|\epsilon_1 v_1 + \cdots + \epsilon_n v_n\|_{S_{p,r}} \leq (1+\epsilon)^{\theta} n^{1/r}$,  for any $\epsilon_1, \ldots, \epsilon_n$ of modulus 1.  
\end{enumerate}

\end{pro}

\begin{proof}  D. Kutzarova and P. K. Lin \cite{KL} proved that there exists vectors $v_1, \ldots, v_n$ in $c_{00}$ which are disjointly supported such that $\|v_1+ \cdots +v_n\|_{S}\leq (1+\epsilon)$, where each $v_j$ is of the form $\frac{m_j}{\log_2(m_j+1)}\sum_{i\in M_j} e_i$, $|M_j|=m_j$.  Since the  basis of $S$ is 1-unconditional, we have $\|\epsilon_1 v_1 + \cdots + \epsilon_n v_n\|_S \leq (1+\epsilon)$ for any $(\epsilon_j)_j^n$ of modulus 1.  

Let $v=  \epsilon_1v_1+ \cdots +\epsilon_nv_n$, it follows from the Lozanovskii  formula that 
$$ \|v\|_{S_{p,r}} \leq \|v\|^{1-\theta}_t \|v\|^{\theta}_S \leq  (1+\epsilon)^{\theta}n^{(1-\theta)/t} = (1+\epsilon)^{\theta}n^{1/r}.
$$
\end{proof}

\begin{pro} Let $1\leq p < r \leq \infty$. The family of Schlumprecht type spaces $S_{p,r}$ and  their  duals are not near Hilbert.  In particular, they are  ergodic spaces.
\end{pro}

\begin{proof}Let $1\leq p < r \leq \infty$. Assume that $S_{p,r}$ has type $t$ then  by Proposition \ref{ScFamily},   $n^{1/p}\log_2(n+1)^{1/r-1/p}\leq T_t n^{1/t}$ for some constant $T_t$, and then $t \leq p$. Hence $p(S_{p,r})\leq p$.   Analogously, if $S_{p,r}$ has cotype $t$, then by Propsition \ref{ScCotype} $t\geq r$, and then $q(S_{p,r})\geq r$. We have that $p(S_{p,r})\leq p<r\leq q(S_{p,r})$ and it follows that $S_{p,r}$ is not near Hilbert.  Also, since a Banach space $X$ is near Hilbert if and only if $X^*$ is near Hilbert, the dual space $S_{p,r}^*$ is not near Hilbert.
\end{proof}

\section{Final Remarks}

Of course, the main question concerning ergodic spaces is whether $\ell_2$ is the only non ergodic Banach space.  
The conclusion of Theorem \ref{main} restricts the question of ergodicity to the case of near Hilbert spaces, but our technique uses a reduction  throughout subspaces without AP.    A Banach space in which all of its subspaces have AP  is said  to have the  {\it hereditarily approximation property} (HAP).
Szankowski \cite{S}  proved that every HAP space must be near Hilbert.  The first example of a HAP space not isomorphic to a Hilbert space was constructed by Johnson \cite{J2}. Later,  Pisier  \cite{P}  proved that every weak Hilbert space  has the HAP.  The space constructed by Johnson is asymptotically Hilbertian, and therefore ergodic by the  Anisca \cite{A}  result. In 2010 Johnson and Szankowski \cite{JS2} constructed a HAP space with a symmetric basis but not isomorphic to $\ell_2$, and hence not asymptotically Hilbertian.
Hence, a natural question is the following:

\

{\bf Problem.} Is the  HAP non asymptotically Hilbertian space contructed in \cite{JS2} ergodic?

\

Or more generally:

\

{\bf Problem.}  Is every HAP not isomorphic to the Hilbert space ergodic?

\

Another interesting class  of near Hilbert spaces are the  twisted Hilbert spaces. The most important example of non trivial twisted Hilbert space is  the Kalton-Peck space $Z_2$ \cite{KP}.  $Z_2$ is not asymptotically Hilbertian and it is not known whether has HAP.

\

{\bf Problem.} Does  there exist an ergodic non trivial twisted Hilbert space?

\

Another natural question is:

\

{\bf Problem.} Is every minimal Banach space not isomorphic to a Hilbert space ergodic?

\section{Acknowledgments}

The author thanks Valentin Ferenczi   for his  helpful discussions and  comments on this paper;  Razvan Anisca  and Bill Johnson for their  remarks  and suggestions.





\address
\end{document}